\documentclass[12pt]{elsarticle}
  
\usepackage{graphics,graphicx,subfigure,amssymb,color}
\usepackage{amsmath}
\usepackage{amsthm}

\newtheorem{thm}{Theorem}[section]

\newtheorem{prop}[thm]{Proposition}

\newtheorem{remark}[thm]{Remark}

\newcommand{\R}{\mathbb R}

\newcommand{\Z}{\mathbb Z}
\newcommand{\N}{\mathbb N}
\newcommand{\C}{\mathbb C}

\journal{}

\begin{document}

\begin{frontmatter}

\title{ On off-critical zeros of lattice energies in the neighborhood of the Riemann zeta function }
\author{Laurent B\'{e}termin}
\address{Institut Camille Jordan, Universit\'e Claude Bernard Lyon 1,
69622 Villeurbanne, France} 

\author{Ladislav \v{S}amaj}
\author{Igor Trav\v{e}nec} 
\address{Institute of Physics, Slovak Academy of Sciences, 
D\'ubravsk\'a cesta 9, 84511 Bratislava, Slovakia}

\begin{abstract}
The Riemann zeta function $\zeta(s):= \sum_{n=1}^{\infty} 1/n^s$
can be interpreted as the energy per point of the lattice $\Z$,
interacting pairwisely via the Riesz potential $1/r^s$.
Given a parameter $\Delta\in (0,1]$, this physical model is generalized
by considering the energy per point $E(s,\Delta)$ of a periodic
one-dimensional lattice alternating the distances between the
nearest-neighbour particles as $2/(1+\Delta)$ and $2\Delta/(1+\Delta)$,
keeping the lattice density equal to one independently of $\Delta$.
This energy trivially satisfies $E(s,1)=\zeta(s)$ at $\Delta=1$,
it can be easily expressed as a combination of the Riemann and
Hurwitz zeta functions, and extended analytically to the punctured
$s$-plane $\C \setminus \{ 1\}$.  
In this paper, we perform numerical investigations of the zeros of
the energy $\{ \rho=\rho_x+{\rm i}\rho_y\}$, which are defined by
$E(\rho,\Delta)=0$. 
The numerical results reveal that in the Riemann limit $\Delta\to 1^-$
theses zeros include the anticipated critical zeros of the Riemann zeta
function with $\Re(\rho_x)=\frac{1}{2}$ as well as an unexpected
-- comparing to the Riemann Hypothesis --  infinite series of off-critical
zeros.
The analytic treatment of these off-critical zeros shows that their
imaginary components are equidistant and their real components diverge
logarithmically to $-\infty$ as $\Delta\to 1^-$, i.e., they become invisible
at the Riemann's $\Delta=1$.
\end{abstract}


\end{frontmatter}

\section{Introduction and main results} \label{Sec1}
Let two  points  at distance $r$ interact via the Riesz potential 
$1/r^s$ with real $s$ \cite{Brau}. 
If the points are located on the lattice $\Z$ and interact pairwisely
by the Riesz potential  where $s>1$, the energy per point
is given by the Riemann zeta function \cite{Riemann1859}
\begin{equation} \label{Riemann}
\zeta(s) := \frac{1}{2} \sum_{n\in \Z^*} 
\frac{1}{\vert n\vert^s} = \sum_{n=1}^{\infty} \frac{1}{n^s}  \qquad s>1 ,
\end{equation}
where the prefactor $\frac{1}{2}$ is due to the fact that each interaction
energy is shared by a pair of  points.
The function  $\zeta$  can be analytically continued to the whole complex
$s$-plane, with a simple pole at $s=1$.
The Riemann zeta function plays a fundamental role in the algebraic and
analytic number theories
\cite{Hadamard93,Hardy14,Riesz16,Hardy21,Hutchinson25,Titchmarsh35,Selberg46},
see monographs \cite{Edwards74,Ivic85,Titchmarsh88}.
The  so-called Riemann Hypothesis about the location of its nontrivial zeros
exclusively on the critical line $\Re(s)=\frac{1}{2}$ (the symbol $\Re$ means 
the real part) is one of the Hilbert and Clay Millennium Prize problems
\cite{Jaffe06}.
Throughout the present paper we assume that the Riemann Hypothesis holds.
The Riemann zeta function and its Epstein's
\cite{Epstein03,Epstein07,Chowla49,Travenec22},
Hurwitz's \cite{Hurwitz1882,Spira76,Nakamura16}, Barnes's \cite{Barnes04}, etc.
generalisations have numerous applications both in mathematics
(prime numbers, applied statistics \cite{Borwein13}) and
in physics \cite{Elizalde12}.

\medskip

Let the Riemann zeta function be a member of a family of functions which
exhibit nontrivial zeros off the critical line.
Possible  mechanisms of the disappearance of these off-critical zeros
at the Riemann's point might of general interest.
In this paper, we propose a natural extension of the Riemann zeta function
as  the energy of a unit density lattice $L_\Delta$ with alternating distances
between the nearest neighbours, say $2/(1+\Delta)$ and
$2\Delta/(1+\Delta)$; due to the $\Delta\to 1/\Delta$ symmetry of
the problem, it is sufficient to restrict oneself to $\Delta$
from the interval $(0,1]$.

\medskip

In analogy with the original model with constant unit spacing, each  point 
interacts pairwisely with  the other points  via the Riesz interaction
$1/r^s$, $s>1$ and the lattice energy per point is therefore given
(see Proposition \ref{prop:energy}) by
$$
E(s,\Delta) = \frac{1}{2^s} \zeta(s) + \frac{1}{2^{s+1}}
\left[ \zeta\left(s,\frac{1}{1+\Delta}\right)
+\zeta\left(s,\frac{\Delta}{1+\Delta}\right) \right],\ s>1, \Delta\in (0,1],
$$
where
\begin{equation} \label{Hurwitz}
\zeta(s,a) = \sum_{n=0}^{\infty} \frac{1}{(n+a)^s} , \qquad s > 1,
\end{equation}
is the Hurwitz zeta function with the real (positive) parameter $a$.
Remark that this lattice energy, as a combination of Riemann and
Hurwitz zeta functions, has an analytic continuation on
$\C\backslash \{1\}$ (see Proposition \ref{prop:analytic}).
For given $\Delta\in (0,1]$, the set of zeros of the lattice energy
is defined as
$Z_\Delta:=\{\rho=\rho_x+{\rm i}\rho_y\in \C ,(\rho_x,\rho_y)\in \R^2 :
E(\rho,\Delta)=0\}$, noticing that $Z_1$ is the set of zeros of
the Riemann zeta function composed by trivial zeros (i.e., $\rho\in -2\N$)
and critical zeros (i.e., $\Re(\rho)=\frac{1}{2}$) assuming that
Riemann Hypothesis holds.
Furthermore, for specific values of the parameter
$\Delta\in\{1/5,1/3,1/2\}$, the energy can be factorized as
$E(s,\Delta)=f_\Delta(s)\zeta(s)$ where $f_\Delta$ is a sum of $p^s$ with
integers $p$.
This automatically gives us critical zeros (i.e. solutions of $\zeta(\rho)=0$
assuming the Riemann Hypothesis) and possible off-critical zeros
(i.e. solutions of $f_\Delta(\rho)=0$), as shown in
Proposition \ref{prop:Zspecialvalues}.

\medskip

The goal of this paper is to study, both numerically and analytically,
the set of zeros $Z_\Delta$ when $\Delta$ is in a neighborhood of $1$,
i.e. when $E(\cdot,\Delta)$ is in the neighborhood of
the Riemann zeta function.
Numerical and analytic analysis shows that approaching $\Delta\to 1^-$ the
zeros of $E(\rho,\Delta)$ involve the anticipated critical zeros of the Riemann
zeta function with $\Re(\rho_x)=\frac{1}{2}$ as well as an infinite series
of unexpected off-critical zeros with the following asymptotics for their
components, as $\Delta\to 1^-$ (see Theorem \ref{thm:asymptzeros}):
\begin{align*}
&\rho_x(\Delta)= \frac{2}{\ln 2} \ln(1-\Delta)+O(1-\Delta)\to -\infty , \\
&\rho_y(\Delta)= \frac{(2k+1)\pi}{\ln 2}+O\left((1-\Delta)^{\frac{2\ln 3}{\ln 2}}
\right), \quad k\in \Z.
\end{align*}
This means that, asymptotically, there is a infinite sequence of
equidistant zero components along the $\rho_y$ axis.
Furthermore, the divergence of $\rho_x$ to $-\infty$ as $\Delta\to 1^-$
is an example of the disappearance of off-critical zeros when approaching
the Riemann's point.
Moreover, the behavior of these zero components with respect to
$\Delta\in (0,1]$ is numerically studied (see Figures \ref{roy} and \ref{rox}).

\medskip

\textbf{Plan of the paper.} 
The generalized 1D model for Riesz points with alternating lattice
spacings is presented in section \ref{Sec2}. 
The energy per particle $E(s,\Delta)$ is expressed as a combination
of Hurwitz zeta functions in section \ref{Sec21}.
The properties of the Hurwitz zeta function are discussed in section
\ref{Sec22}.
Special values of the parameters $\Delta$ when the energy $E(s,\Delta)$
factorizes itself onto the product of the Riemann zeta function and
some simple function are given in section \ref{Sec23}.
Numerical results for zeros at any $0<\Delta<1$, together with tests
at the special values of $\Delta=1/5,1/3,1/2$ are presented in section
\ref{Sec3}.
The spectrum of critical and off-critical zeros in the Riemann's
limit $\Delta\to 1^-$ is discussed in section \ref{Sec4}.

\section{The generalized one-dimensional model} \label{Sec2}
\subsection{Definition of the model} \label{Sec21}

Given $\Delta\in (0,1]$, we consider the infinite set of points
$L_\Delta\subset \R$ given by
$$
L_\Delta:=2\Z \cup \left( 2\Z + \frac{2\Delta}{1+\Delta}\right),
$$
which is the unit density periodic configuration with alternating distances $2/(1+\Delta)$
and $2\Delta/(1+\Delta)$, since $2-2/(1+\Delta)=2\Delta/(1+\Delta)$. Assuming that each pair of points in $L_\Delta$ interacts via the Riesz potential $1/r^s$, $s>1$, the total energy per point of this system is therefore
$$
E(s,\Delta):=\frac{1}{4}\sum_{k\in \left\{0,\frac{2\Delta}{1+\Delta} \right\}}\sum_{p\in L_\Delta \atop p\neq k } \frac{1}{|p-k|^s}.
$$
The following proposition shows how to write this energy in terms of Riemann and Hurwitz zeta functions.
\begin{prop}\label{prop:energy}
For any $s>1$ and any $\Delta\in (0,1]$, we have
\begin{equation}\label{elin}
E(s,\Delta) = \frac{1}{2^s} \zeta(s) + \frac{1}{2^{s+1}}
\left[ \zeta\left(s,\frac{1}{1+\Delta}\right)
+\zeta\left(s,\frac{\Delta}{1+\Delta}\right) \right].
\end{equation}
\end{prop}
\begin{proof}
We simply compute the above double sum as follows:
\begin{align*}
E(s,\Delta):&= \frac{1}{4}\sum_{k\in \left\{0,\frac{2\Delta}{1+\Delta}\right\}}
\sum_{p\in L_\Delta \atop p\neq k } \frac{1}{|p-k|^s} \\
& =\frac{1}{4}\sum_{k\in \left\{0,\frac{2\Delta}{1+\Delta}\right\}}
\left(\sum_{p\in 2\Z \atop p\neq k } \frac{1}{|p-k|^s} +
\sum_{p\in2\Z+ \frac{2\Delta}{1+\Delta} \atop p\neq k } \frac{1}{|p-k|^s} \right) \\
& =\frac{1}{4}\sum_{n\in \Z \atop n\neq 0} \frac{1}{|2n|^s}
+\frac{1}{4}\sum_{n\in \Z \atop 2n\neq  \frac{2\Delta}{1+\Delta}}
\frac{1}{\left| 2n- \frac{2\Delta}{1+\Delta} \right|^s} \\
& + \frac{1}{4}\sum_{n\in \Z \atop 2n\neq - \frac{2\Delta}{1+\Delta}}
\frac{1}{\left| 2n+ \frac{2\Delta}{1+\Delta} \right|^s}
+\frac{1}{4}\sum_{n\in \Z \atop n\neq 0} \frac{1}{|2n|^s}\\
& =\frac{1}{2^{s+1}}\sum_{n\in \Z^*} \frac{1}{|n|^s}
+\frac{1}{2^{s+2}}\sum_{n\in \Z}
\frac{1}{\left| n+ \frac{\Delta}{1+\Delta} \right|^s}+\frac{1}{2^{s+2}}
\sum_{n\in \Z}\frac{1}{\left| n- \frac{\Delta}{1+\Delta} \right|^s}.
\end{align*}
We now split the two last sums in order to get two Hurwitz zeta functions and two rests that we write again in terms of the same Hurwitz zeta functions:
\begin{align*}
E(s,\Delta):&=\frac{1}{2^{s}}\zeta(s)
+\frac{1}{2^{s+2}}\left[ \zeta\left(s,\frac{1}{1+\Delta}\right)
+\zeta\left(s,\frac{\Delta}{1+\Delta}\right)\right] \\
&\qquad+\frac{1}{2^{s+2}}\left(\sum_{n=1}^{+\infty}
\frac{1}{\left| -n+ \frac{\Delta}{1+\Delta} \right|^s}
+\sum_{n=1}^{+\infty}\frac{1}{\left| -n+ \frac{1}{1+\Delta} \right|^s} \right)\\
&=\frac{1}{2^{s}}\zeta(s)+\frac{1}{2^{s+2}}
\left[ \zeta\left(s,\frac{1}{1+\Delta}\right)
+\zeta\left(s,\frac{\Delta}{1+\Delta}\right)\right] \\
& \qquad+\frac{1}{2^{s+2}}\left(\sum_{n=1}^{+\infty}
\frac{1}{\left( n- \frac{\Delta}{1+\Delta} \right)^s}
+\sum_{n=1}^{+\infty}\frac{1}{\left( n- \frac{1}{1+\Delta} \right)^s} \right).
\end{align*}
Since we have, by the change of variables $n=k+1$,
$$
\sum_{n=1}^{+\infty}\frac{1}{\left( n- \frac{\Delta}{1+\Delta} \right)^s}
=\sum_{k=0}^{+\infty} \frac{1}{\left(k+\frac{1}{1+\Delta} \right)^s}
\ \textnormal{and}\
\sum_{n=1}^{+\infty}\frac{1}{\left( n- \frac{1}{1+\Delta} \right)^s}
=\sum_{k=0}^{+\infty} \frac{1}{\left(k+\frac{\Delta}{1+\Delta} \right)^s},
$$
we obtain
\begin{align*}
E(s,\Delta)&=\frac{1}{2^{s}}\zeta(s)+\frac{2}{2^{s+2}}
\left[ \zeta\left(s,\frac{1}{1+\Delta}\right)
+\zeta\left(s,\frac{\Delta}{1+\Delta}\right)\right] \\
&=\frac{1}{2^{s}}\zeta(s)+\frac{1}{2^{s+1}}
\left[ \zeta\left(s,\frac{1}{1+\Delta}\right)
+\zeta\left(s,\frac{\Delta}{1+\Delta}\right)\right]
\end{align*}
and the proof is complete.
\end{proof}
Notice that the energy satisfies the required symmetry relation
$E(s,\Delta) = E(s,1/\Delta)$.

\begin{remark}[\textbf{Crystallization result}]
It is known (see e.g. {\rm \cite{Ventevogel}}), by a convexity argument
(or by the so-called one-dimensional ``Universal Optimality" of $\Z$,
see {\rm \cite{CohnKumar}}) that, for all $s>0$,
$$
\min_{\Delta\in (0,1]} E(s,\Delta)=E(s,1)=\zeta(s),
$$
with equality if and only if $\Delta=1$. From our results (see Theorem \ref{thm:asymptzeros}), the Riemann zeta function is therefore at the same time the minimal value of our
energy and the only one in its $\Delta$-neighborhood for which the non-trivial zeros are strictly
located on the critical line $\Re(s)=\frac{1}{2}$. It might be interesting to investigate other lattice energies to understand how universal this phenomenon is.
\end{remark}

\subsection{The Hurwitz zeta function and the analytic continuation of $E(\cdot, \Delta)$}\label{Sec22}
The Hurwitz zeta function (\ref{Hurwitz}) is a generalization of
the Riemann zeta-function (\ref{Riemann}) via a shift $a>0$.
In particular,
\begin{equation}
\zeta(s,1) = \zeta(s).
\end{equation}
In the symbolic computer language {\it Mathematica}, the Riemann
and Hurwitz zeta functions are tabulated under the symbols Zeta[$s$]
and Zeta[$s,a$], respectively.

The Hurwitz zeta function satisfies two easily verifiable important equalities,
\begin{equation} \label{rel1}
\forall x\in \left[ 0, \frac{1}{2}\right], \forall s>1,
\quad \zeta(s,x) + \zeta(s,1/2+x) = 2^s \zeta(s,2x) ,
\end{equation}  
and the multiplication theorem
\begin{equation}\label{relgeneral}
k\in \N,\forall s>1, \quad k^s \zeta(s)
=\sum_{n=1}^k\zeta\left( s, \frac{n}{k}\right).
\end{equation}
The second relation easily implies that
\begin{equation} \label{rel2}
\zeta(s,1/3) + \zeta(s,2/3) = (3^s-1) \zeta(s)
\end{equation}
as well as
\begin{equation} \label{half}
\zeta(s,1/2) = (2^s-1) \zeta(s).
\end{equation}
From \eqref{rel1} with $x=1/4$ and \eqref{half}, we therefore obtain
\begin{equation}\label{rel2bis}
\zeta(s,1/4)+\zeta(s,3/4)=(4^s-2^s)\zeta(s).
\end{equation}

From \eqref{elin} and \eqref{half}, it is also straightforward to check
that $E(s,1)=\zeta(s)$.

\medskip

Furthermore, it is clear that \eqref{elin} holds for all $s\in \C$ such
that $\Re(s)>1$ and we get the following result by directly applying
the classical one (see e.g. \cite{Fine51,Hurwitz1882,Travenec22})
on the analytic continuation of the Riemann and Hurwitz zeta functions.

\begin{prop}\label{prop:analytic}
For all $\Delta\in (0,1]$, the function $s\mapsto E(s,\Delta)$ has
an analytic continuation on $\C\backslash \{1\}$.
Furthermore, we have, for all $s$ such that $\Re(s)<1$ and all
$\Delta\in (0,1]$, 
\begin{align*}
&\pi^{-\frac{s}{2}}\Gamma\left(\frac{s}{2} \right)E(s,\Delta) \\
&= \frac{1}{2^{s+1}}\left\{
\int_0^\infty \left[\vartheta\left(\frac{1}{1+\Delta},it\right)-1\right]
t^{\frac{-1-s}{2}} {\rm d}t + f(s) \right\} ,
\end{align*}
where $\displaystyle \vartheta(z,it)=\sum_{n\in \Z} e^{-\pi n^2 t}e^{2i\pi n z}$
is the Jacobi theta function defined for $t>0$ and $z\in \C$, and
\begin{align*}
f(s) &= \int_0^\infty \left[ \vartheta(0,it)-1-\frac{1}{\sqrt{t}} \right]
t^{\frac{s}{2}-1}{\rm d}t \qquad \mbox{for $0<\Re(s)<1$,} \\
&= \int_0^\infty \left[ \vartheta(0,it)-\frac{1}{\sqrt{t}} \right]
t^{\frac{s}{2}-1}{\rm d}t \qquad \mbox{for $\Re(s)<0$.}
\end{align*}
\end{prop}
\begin{proof}
Recall that, for $\Delta\in (0,1]$ and $\Re(s)>1$, we have
$$
E(s,\Delta) = \frac{1}{2^s} \zeta(s) + \frac{1}{2^{s+1}}
\left[ \zeta\left(s,\frac{1}{1+\Delta}\right)
+\zeta\left(s,\frac{\Delta}{1+\Delta}\right) \right].
$$
It has been shown in \cite{Fine51,Hurwitz1882,Travenec22} that
$s\mapsto \zeta(s)$ and $s\mapsto \zeta(s,a)$, $a>0$, admit an analytic
continuation to $\C\backslash \{1\}$, which implies the same for
$s\mapsto E(s,\Delta)$.
Furthermore, writing $z=\frac{1}{1+\Delta}$, we have
$$
E(s,\Delta)= \frac{1}{2^s} \zeta(s) + \frac{1}{2^{s+1}}
\left[ \zeta\left(s,z\right)
+\zeta\left(s,1-z\right) \right] .
$$
Considering the analytic continuation of $\zeta(1-\alpha,a)$, $a>0$,
the following formula is well-known \cite{Fine51,Hurwitz1882} for all
$\alpha$ such that $\Re(\alpha)>0$: 
\begin{equation}\label{eq:zetatheta}
\pi^{-\frac{1-\alpha}{2}}\Gamma\left( \frac{1-\alpha}{2}\right)
\left[ \zeta\left(1-\alpha,z\right)
+\zeta\left(1-\alpha,1-z\right) \right]
=\int_0^\infty[\vartheta(z,it)-1]t^{\frac{\alpha}{2}}\frac{{\rm d}t}{t}
\end{equation}
holding for $z\notin \Z$ (see the remark in Theorem 12.6 on page 257
of \cite{Apostol76}).
Replacing $\alpha$ by $s=1-\alpha$ (so that $\Re(s)<1$ for $\Re(\alpha)>0$),
we write 
\begin{align*}
\pi^{-\frac{s}{2}}\Gamma\left(\frac{s}{2} \right)E(s,\Delta)
&= \frac{\pi^{-\frac{s}{2}}}{2^s}\Gamma\left(\frac{s}{2} \right)\zeta(s)
+ \frac{\pi^{-\frac{s}{2}}}{2^{s+1}}\Gamma\left(\frac{s}{2} \right)
\left[ \zeta\left(s,z\right) +\zeta\left(s,1-z\right) \right]\\
&=\frac{2\pi^{-\frac{s}{2}}}{2^{s+1}}\Gamma\left(\frac{s}{2} \right)\zeta(s)
+ \frac{\pi^{-\frac{s}{2}}}{2^{s+1}}\Gamma\left(\frac{s}{2} \right)
\left[ \zeta\left(s,z\right) +\zeta\left(s,1-z\right) \right] .
\end{align*}
Next we replace
$\pi^{-\frac{s}{2}}\Gamma\left(\frac{s}{2} \right)\left[ \zeta\left(s,z\right)
+\zeta\left(s,1-z\right) \right]$
by the integral given in \eqref{eq:zetatheta} and
$2\pi^{-\frac{s}{2}}\Gamma\left(\frac{s}{2} \right)\zeta(s)$
by the $d=1$ integral in Eq. (21) of \cite{Travenec22} for $0<\Re(s)<1$ and
by the $d=1$ integral in Eq. (22) of \cite{Travenec22} for $\Re(s)<0$
to complete the proof.
\end{proof}
Therefore, we can consider the zeros of $s\mapsto E(s,\Delta)$
in $\C\backslash \{1\}$ defined as
$$
Z_\Delta:=\{\rho=\rho_x+{\rm i}\rho_y\in \C ,(\rho_x,\rho_y)\in \R^2 : E(\rho,\Delta)=0\},
$$
noticing that $Z_1$ is the set of zeros of the Riemann zeta function.
We recall that, according to the Riemann Hypothesis,
$$
Z_1=-2\N\cup Z^C, \quad Z^C\subset \{\Re(z)=1/2\},
$$
where $Z^C$ is called the set of critical zeros of $\zeta$ and $-2\N$ are the trivial zeros of $\zeta$. We are going to see in the next sections that $Z_\Delta$ can have other nontrivial off-critical zeros.

\subsection{Factorization and zeros of the energy for special values of $\Delta$}
\label{Sec23}
There exist special values of $\Delta$ for which the energy $E(s,\Delta)$
factorizes itself into a product of the Riemann zeta function $\zeta(s)$ and
some simple functions of $s$, by using the previously presented relations
(\ref{rel1}) and (\ref{rel2}).
For these cases, both critical and off-critical zeros can be found easily.
The most obvious choice of $\Delta$ is $\Delta=1$ for which we have
\begin{equation} \label{d1}
E(s,1)=\zeta(s).
\end{equation}
In the cases $\Delta\in \left\{\frac{1}{2}, \frac{1}{3} \right\}$ we have the following result giving the zeros of $E(s,\Delta)$ as well as the factorization of the energy.
\begin{prop}\label{prop:Zspecialvalues}
For all $s\in \C\backslash \{1\}$, we have
\begin{equation}\label{eq:facto1213}
E(s,1/2)=\frac{1}{2^{s+1}}(1+3^s)\zeta(s) \quad \textnormal{and}\quad
E(s,1/3)=\frac{1}{2^{s+1}}(2-2^s+4^s) \zeta(s) .
\end{equation}
Furthermore, the zeros of $E(s,1/2)$ and $E(s,1/3)$ are 
\begin{align*}
&Z_{\frac{1}{2}}=Z_1 \cup
\left\{\frac{(2k+1) {\rm i} \pi}{\ln 3}\right\}_{k\in \Z },\\
& Z_{\frac{1}{3}}=Z_1\cup \left\{\frac{1}{\ln{2}}
\left[ \ln \left( \frac{1\pm {\rm i}\sqrt{7}}{2}\right)
+2 \pi {\rm i} k\right] \right\}_{k\in \Z}.
\end{align*}
\end{prop}
\begin{proof}
For $\displaystyle\Delta=\frac{1}{2}$, one has
$\displaystyle\frac{1}{1+\Delta}=\frac{2}{3}$ and
$\displaystyle\frac{\Delta}{1+\Delta}=\frac{1}{3}$ and therefore,
using \eqref{rel2}, it holds that
\begin{equation} \label{d2}
E(s,1/2)=\frac{1}{2^{s+1}}(1+3^s)\zeta(s).
\end{equation}
where the function $1+3^s$ yields an infinite sequence of (purely imaginary)
off-critical zeros
\begin{equation} \label{zeros2}
\rho_k = \frac{(2k+1){\rm i}\pi}{\ln{3}}, \quad k\in \Z.
\end{equation}
Furthermore, for $\displaystyle\Delta=\frac{1}{3}$, we have
$\displaystyle\frac{1}{1+\Delta}=\frac{3}{4}$ and
$\displaystyle\frac{\Delta}{1+\Delta}=\frac{1}{4}$ and therefore,
applying \eqref{rel2bis}, it holds that
\begin{equation} \label{d3}
E(s,1/3)=\frac{1}{2^{s+1}}(2-2^s+4^s) \zeta(s).
\end{equation}
The function $2-2^s+4^s$ yields the following zeros
\begin{equation} \label{zeros3}
\rho_k=\frac{1}{\ln{2}} \left[ \ln \left( \frac{1\pm {\rm i}\sqrt{7}}{2}\right)
+2 \pi {\rm i} k\right] , \quad k\in \Z.
\end{equation}
\end{proof}
\begin{remark}
In the $\Delta=1/3$ case, since all these zeros have the real part equal to $\frac{1}{2}$, the energy for $\Delta=1/3$ exhibits only critical zeros.
\end{remark}

The last factorization we are considering in our paper corresponds to $\Delta=1/5$:
\begin{equation} \label{d5}
E(s,1/5)=\frac{1}{2^{s+1}}(3-2^s-3^s+6^s)\zeta(s).
\end{equation}
The function $3-2^s-3^s+6^s$ exhibits only off-critical zeros which can be
found only numerically, e.g., $s\approx 0.635084\pm 1.07885{\rm i}$.

\section{Numerical results} \label{Sec3}
The starting point of our numerical determination of zeros of the energy
$E(s,\Delta)$ was the case $\Delta=1/2$, with the factorization form \eqref{eq:facto1213},
whose spectrum of zeros involves both the critical zeros of the Riemann
zeta function as well as an infinite set of off-critical zeros (\ref{zeros2}).
It was checked that the accuracy of determination of complex zeros by using
the symbolic language {\it Mathematica} is 34-35 decimal digits for both real
and imaginary components.

Then we proceeded to the left and right from this point $\Delta=1/2$ by
changing successively $\Delta$ by a small amount to avoid an uncontrolled
skip between neighbouring branches of zeros.
Our numerical experience indicates that changing $\Delta$ by 0.01 is
certainly safe from this point of view. 
We observe the following:
\begin{itemize}
\item The zeros of $E(s,\Delta)$ form continuous non-crossing curves in the
complete $(\rho_x, \rho_y, \Delta)$ space. 
\item Nevertheless, the curves may intersect in the reduced spaces
$(\rho_y,\Delta)$ and $(\rho_x,\Delta)$, see Figures \ref{roy} and \ref{rox},
respectively.
\end{itemize}
Notice that the values of $\Delta\in \{1/5,1/3,1/2\}$ serve as
test points of our numerical calculations.

\begin{figure}[t]
\centering
\includegraphics[clip,width=0.75\textwidth]{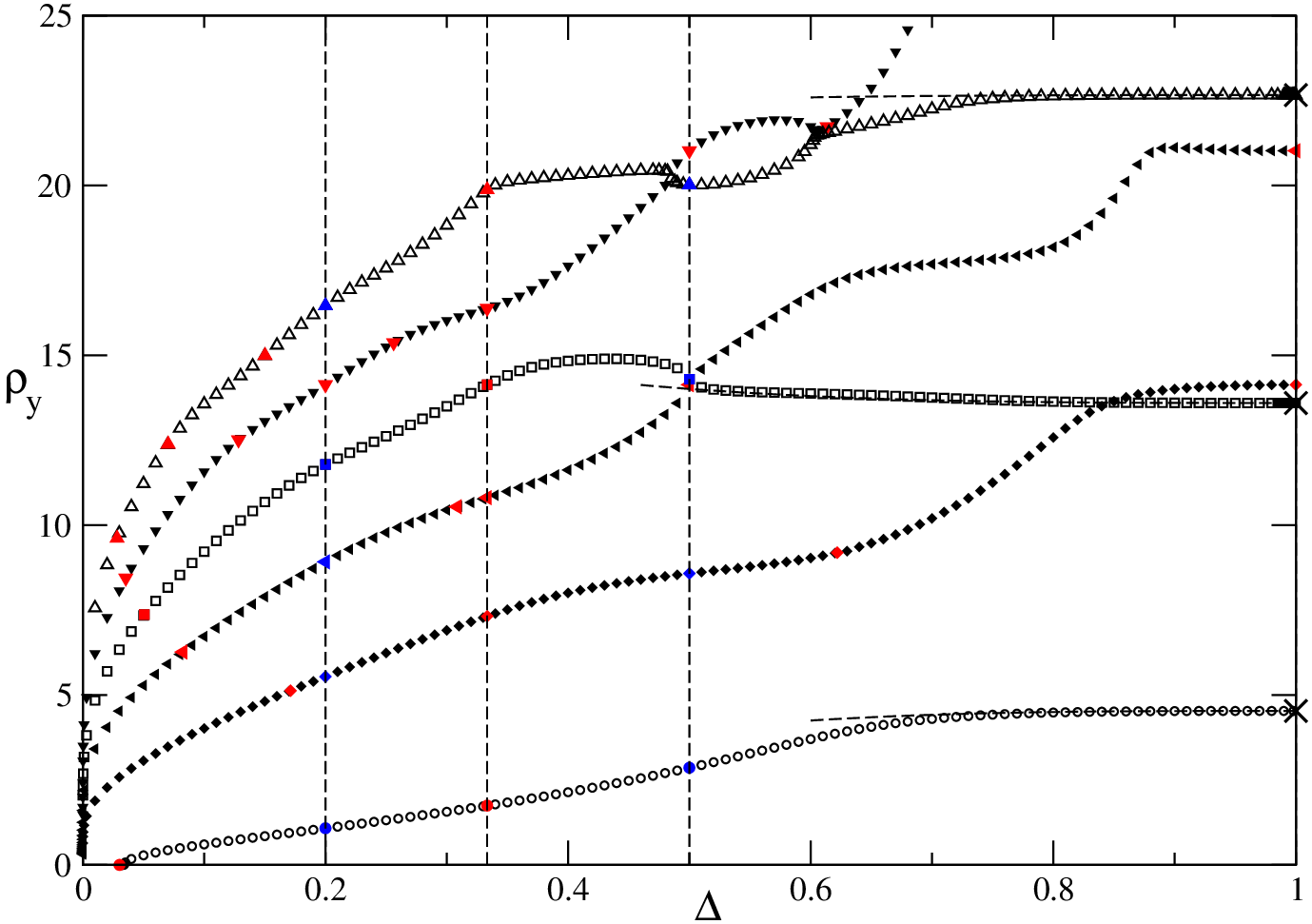}
\caption{Imaginary parts $\rho_y$ of the energy zeros $E(\rho,\Delta)=0$.
The special cases $\Delta\in \{1/5,1/3,1/2,1\}$ are visualized by vertical
dashed lines.
The points which correspond to critical zeros with $\rho_x=\frac{1}{2}$
are denoted by red colour.}
\label{roy}
\end{figure}

The dependence of the imaginary component of zeros $\rho_y$, in the range
of its values $[0,25]$, on the parameter $\Delta\in (0,1]$ is pictured
in Figure~\ref{roy}.
The special cases $\Delta\in \{1/5,1/3,1/2\}$, when the energy factorizes
itself onto a product of the Riemann zeta function and a simple function,  
are visualized by vertical dashed lines.
These cases yield us precise values of the corresponding zeros and also make
us sure not to miss any curve of zeros. 
The critical zeros with $\rho_x=\frac{1}{2}$ are denoted by red colour,
all other zeros are off-critical; the off-critical zeros for the special
values of $\Delta\in \{1/5,1/2\}$ are denoted by blue colour. We observe two types of zero curves:
\begin{itemize}
\item The first three ``standard'' zero curves $\rho_y(\Delta)$, which end up at
the Riemann critical zeros at $\Delta=1$, are represented by full symbols.
\item The first three ``non-standard'' curves $\rho_y(\Delta)$ are represented
by open symbols.
Since in the limit $\Delta\to 1^-$ these curves tend to off-critical zeros
with the divergent real component $\rho_x\to -\infty$, the curves end up with
crosses indicating the absence of off-critical zeros at the Riemann's
$\Delta=1$.
\end{itemize}

\begin{figure}[t]
\centering
\includegraphics[clip,width=0.75\textwidth]{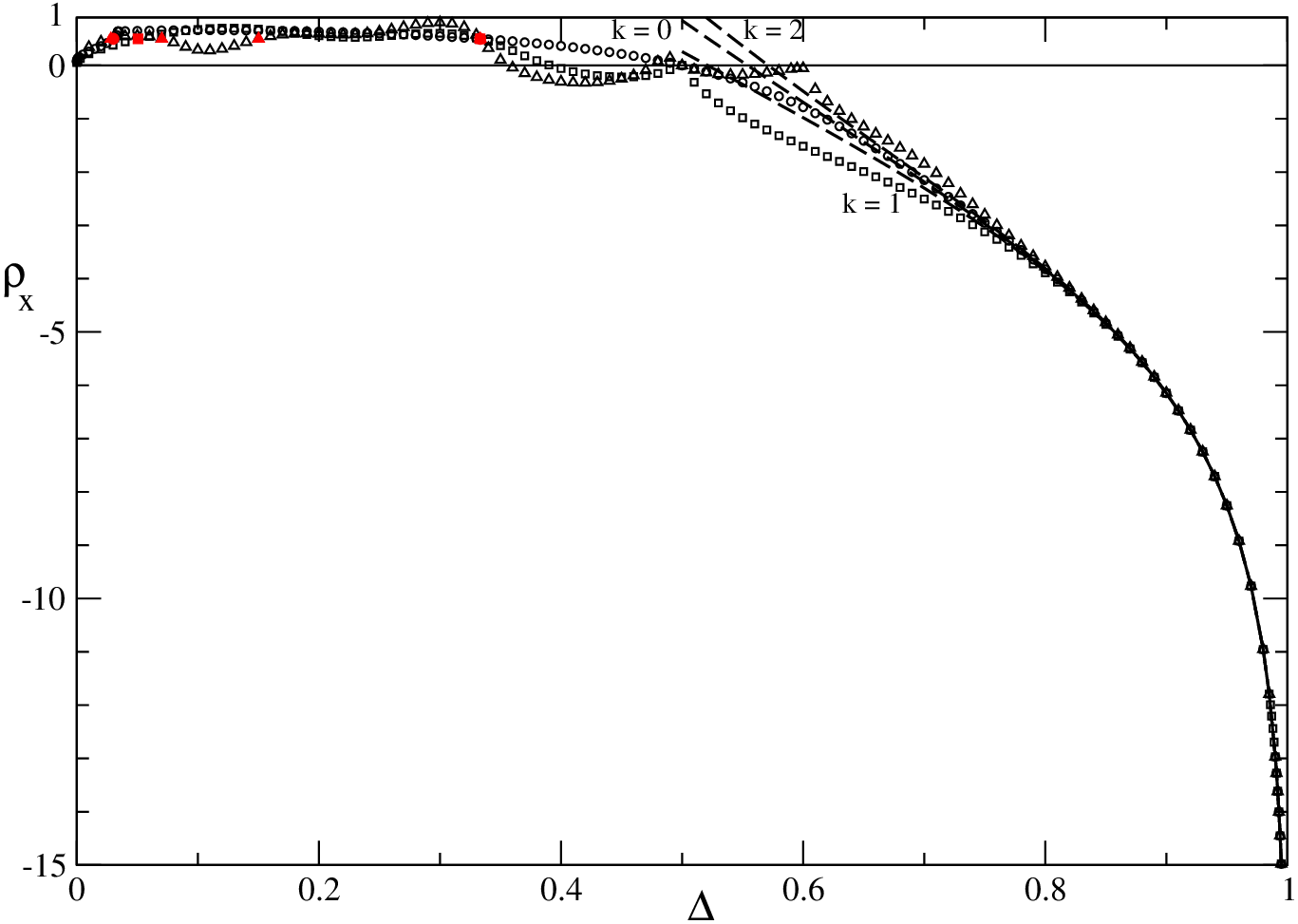}
\caption{The dependence of the real part $\rho_x$ of zeros $E(\rho,\Delta)=0$,
corresponding to the three ``nonstandard''  curves denoted by open symbols
in Figure~\ref{roy}, on the parameter $\Delta\in [0,1]$.
The critical zeros with $\rho_x=\frac{1}{2}$ are denoted by red colour.}
\label{rox}
\end{figure}

The dependence of the real part $\rho_x$ of the first three
``non-standard'' energy zeros on $\Delta\in (0,1]$ is presented
in Figure~\ref{rox} by open symbols, in close analogy with Figure~\ref{roy}.
As before, the critical zeros with $\rho_x=\frac{1}{2}$ are denoted by
red colour. We observe the following:
\begin{itemize}
\item For $\Delta\gtrapprox 0.75$, the three curves coincide
on the considered scale and go to $-\infty$ as $\Delta\to 1^-$.
\end{itemize}

\section{Analytic results in the limit $\Delta\to 1^-$} \label{Sec4}
Approaching $\Delta\to 1^-$, with regard to Eq. (\ref{d1}) one anticipates
the presence of critical zeros of the Riemann zeta function for $E(s,\Delta)$.
Surprisingly, as was already indicated, there are also additional curves of
off-critical zeros.

To derive coordinates of these off-critical zeros, we set
$\Delta=1-\varepsilon$ in (\ref{elin}) and expand the energy in Taylor series
in the small positive $\varepsilon\to 0^+$ up to the order $\varepsilon^5$.

\begin{prop}\label{prop:Taylorepsilon}
Let $s\in \C\backslash \{1\}$, then, as $\varepsilon\to 0^+$,
\begin{align*}
E(s,1-\varepsilon) & = \zeta(s) +  \frac{2^{2+s}-1}{2^{5+s}} s (1+s)
\zeta(2+s) \left( \varepsilon^2 + \varepsilon^3
+ \frac{3}{4} \varepsilon^4 + \frac{1}{2} \varepsilon^5 \right) \\
& + \frac{1}{3} \frac{2^{4+s}-1}{2^{11+s}} s (1+s) (2+s) (3+s)
\zeta(4+s) \left( \varepsilon^4 + 2 \varepsilon^5 \right)
+ O(\varepsilon^6) .
\end{align*}
\end{prop}
\begin{proof}
It directly follows from the Taylor expansion of the Hurwitz zeta function
(see e.g. \cite{Vepstas08}): for $|a|<1$,
\begin{equation}
\zeta(s,a)=\frac{1}{a^s}+\sum_{n=0}^\infty (-a)^n
\binom{s+n-1}{n} \zeta(s+n) ,
\end{equation}
where the binomial coefficient for a complex $s$ has to be understood as
\begin{equation}
\binom{s+n-1}{n} = \frac{s (s+1) (s+2) \cdots (s+n-1)}{n!},
\end{equation}
as well as the analytic continuation of $s\mapsto E(s,\Delta)$.
\end{proof}
It is clear that in the limit $\varepsilon\to 0^+$ the zeros of
$E(s,1-\varepsilon)$ coincide trivially with the critical ones
of the Riemann zeta function $\zeta(s)$.
Let us compute the $\varepsilon\to 0^+$ asymptotics of the other
nontrivial zeros.

\begin{thm}\label{thm:asymptzeros}
The nontrivial off-critical zeros of $E(s,1-\varepsilon)$ are given, as
$\varepsilon\to 0^+$, by $\{ \rho(k)=\rho_x(k)+{\rm i}\rho_y(k)\}_{k\in \Z}$
where
\begin{align}
&\rho_x(k)= \frac{2}{\ln{2}}\ln\varepsilon +
\left( -3 + \frac{2}{\ln 2} \ln \pi \right) + \frac{1}{\ln 2} \varepsilon
\nonumber \\ 
 & + \frac{1}{\ln 2} \left( \frac{1}{4} + \frac{7 \pi^2}{24} \right) \nonumber
\varepsilon^2 + \frac{1}{\ln 2} \left( \frac{1}{12}
+ \frac{7 \pi^2}{24} \right) \varepsilon^3 \nonumber \\
& + \frac{8}{3 \ln 2} \left( \frac{\pi^2}{8} \right)^{\frac{\ln 3}{\ln 2}}  
\cos\left[ \frac{\ln 3}{\ln{2}} (2k+1) \pi \right]
\varepsilon^{2\frac{\ln 3}{\ln 2}}+ o\left(\varepsilon^{2\frac{\ln 3}{\ln 2}} \right),
\label{rhoxk}\\
&\rho_y(k)=\frac{1}{\ln{2}} (2k+1)\pi + \frac{8}{3 \ln 2}
\left( \frac{\pi^2}{8} \right)^{\frac{\ln 3}{\ln 2}}  
\sin\left[ \frac{\ln 3}{\ln{2}} (2k+1) \pi \right]
\varepsilon^{2\frac{\ln 3}{\ln 2}} \\
& + o\left(\varepsilon^{2\frac{\ln 3}{\ln 2}} \right). \label{rhoyk}
\end{align}
In particular,
\begin{enumerate}
\item Vanishing of off-critical zeros: we have
$\displaystyle \lim_{\varepsilon \to 0^+} \rho_x(k)=-\infty$;
\item Asymptotic crystallization of their imaginary parts: at first order,
as $\varepsilon\to 0^+$, the imaginary parts of off-critical zeros are
equidistributed on the lattice $\displaystyle \frac{(2\Z+1)\pi}{\ln 2}$.
\end{enumerate}
\end{thm}
\begin{proof}
The other nontrivial zeros $\{ \rho\}$, besides the one of the Riemann
zeta function, correspond to solutions of the equation
\begin{eqnarray} 
2^{\rho} & = & \frac{1-2^{2+\rho}}{2^5}
\frac{\rho (1+\rho) \zeta(\rho+2)}{\zeta(\rho)}
\left( \varepsilon^2 + \varepsilon^3 + \frac{3}{4} \varepsilon^4
+ \frac{1}{2} \varepsilon^5 \right) \nonumber \\ & &
+ \frac{1}{3} \frac{1-2^{4+\rho}}{2^{11}}
\frac{\rho (1+\rho) (2+\rho) (3+\rho) (4+\rho) \zeta(4+\rho)}{\zeta(\rho)}
\left( \varepsilon^4 + 2 \varepsilon^5 \right)
\nonumber \\ & & + O(\varepsilon^6) . \label{crucial}
\end{eqnarray}  
As will be showed later, the component $\rho_x$ of $\rho=\rho_x+{\rm i}\rho_y$
goes to $-\infty$ as $\varepsilon\to 0^+$.
To simplify our computations, one applies the well known duality transformation
\begin{equation}
\pi^{-s/2} \Gamma\left( \frac{s}{2} \right) \zeta(s) =
\pi^{(s-1)/2} \Gamma\left( \frac{1-s}{2} \right) \zeta(1-s)  
\end{equation}
to each Riemann zeta function in (\ref{crucial}).
%\begin{eqnarray}
%2^{\rho} & = & \frac{1-2^{2+\rho}}{2^5} \frac{\rho (1+\rho) \pi^2
%\Gamma\left(-\frac{1+\rho}{2}\right) \Gamma\left(\frac{\rho}{2}\right)}{
%\Gamma\left(\frac{1-\rho}{2}\right) \Gamma\left(\frac{\rho}{2}+1\right)}
%\frac{\zeta(-1-\rho)}{\zeta(1-\rho)} \left( \varepsilon^2 + \varepsilon^3
%+ \frac{3}{4} \varepsilon^4 + \frac{1}{2} \varepsilon^5 \right) \nonumber \\
%& & + \frac{1}{3} \frac{1-2^{4+\rho}}{2^{11}}
%\frac{\rho (1+\rho) (2+\rho) (3+\rho) \pi^4
%\Gamma\left(-\frac{3+\rho}{2}\right) \Gamma\left(\frac{\rho}{2}\right)}{
%\Gamma\left(\frac{1-\rho}{2}\right) \Gamma\left(\frac{\rho}{2}+2\right)}
%\frac{\zeta(-3-\rho)}{\zeta(1-\rho)} \nonumber \\ & & \times
%\left( \varepsilon^4 + 2 \varepsilon^5 \right) + O(\varepsilon^6) .  
%\end{eqnarray}  
Using then the formula $\Gamma(x+1) = x \Gamma(x)$,
one ends up with the result
\begin{eqnarray} 
2^{\rho} & = & - \frac{1-2^{2+\rho}}{2^3} \pi^2
\frac{\zeta(-1-\rho)}{\zeta(1-\rho)}
\left( \varepsilon^2 + \varepsilon^3 + \frac{3}{4} \varepsilon^4
+ \frac{1}{2} \varepsilon^5 \right) \nonumber \\ & &
+ \frac{1}{3} \frac{1-2^{4+\rho}}{2^7} \pi^4 \frac{\zeta(-3-\rho)}{\zeta(1-\rho)}
\left( \varepsilon^4 + 2 \varepsilon^5 \right) + O(\varepsilon^6) .
\label{cruc}  
\end{eqnarray}
In the limit $\varepsilon\to 0^+$, the r.h.s. of this equation vanishes
and, consequently, the component $\rho_x$ of $\rho=\rho_x+{\rm i}\rho_y$
must go to $-\infty$ as indicated before.
In the limit $\rho_x\to -\infty$, the ratios of Riemann zeta functions
in (\ref{cruc}) can be expanded as follows
\begin{equation}
\frac{\zeta(-1-\rho)}{\zeta(1-\rho)} =
\frac{1+2^{\rho+1} + 3^{\rho+1}+\sum_{k\geq 4}k^{\rho+1}}{1 + 2^{\rho-1}
+ 3^{\rho-1}+\sum_{k\geq 4}k^{\rho-1}} = 1 + \frac{3}{2} 2^{\rho}
+ \frac{8}{3} 3^{\rho} + O(4^{\rho})  
\end{equation}
and
\begin{equation}
\frac{\zeta(-3-\rho)}{\zeta(1-\rho)} =
\frac{1+2^{\rho+3} + 3^{\rho+3}+\sum_{k\geq 4}k^{\rho+3}}{1 + 2^{\rho-1}
+ 3^{\rho-1}+\sum_{k\geq 4}k^{\rho-1}} = 1 + \frac{15}{2} 2^{\rho}
+ O(3^{\rho}) .
\end{equation}

In the leading order of the smallness parameter $\varepsilon$, it holds that
\begin{equation} \label{leading}
2^{\rho_x+{\rm i}\rho_x} = - \frac{\pi^2}{8} \varepsilon^2 +o(\varepsilon^2),
\end{equation}
Since the right-hand side of this equation is real and negative,
the leading order of the $\rho_y$-component is given by
$2^{{\rm i}\rho_y} = -1 + o(1)$, or, equivalently,
\begin{equation} \label{rhoy}
\rho_y(k) = \frac{1}{\ln{2}} (2k+1) \pi + o(1) , \qquad k\in \Z.
\end{equation}
This means that in the limit $\varepsilon\to 0^+$ there exists an infinite
sequence of equidistant zero components along the $\rho_y$ axis. 
As follows from (\ref{leading}), the $x$-component of these zeros diverges
logarithmically as $\varepsilon\to 0^+$:
\begin{equation} \label{rhox}
\rho_x(k) = \frac{2}{\ln{2}}\ln\varepsilon + \left( -3 + \frac{2}{\ln 2}
\ln \pi \right) + O(\varepsilon).
\end{equation}
Note that the leading terms are the same for any value of $k$.
This behavior can be seen in Figure~\ref{rox}.

Higher orders of the expansion of $\rho_y(k)$ and $\rho_x(k)$ 
in $\varepsilon$ can be obtained by inserting the leading order expressions
(\ref{rhoy}) and (\ref{rhox}) directly into the basic relation (\ref{cruc}).
Performing the expansion procedure in $\varepsilon$ it is important to
realize that
\begin{equation}
3^s =
\left( \frac{\pi^2}{8} \varepsilon^2 \right)^{\frac{\ln 3}{\ln 2}}
\exp\left[ {\rm i} \frac{\ln 3}{\ln 2} (2k+1) \pi \right]
+ o\left(\varepsilon^{2\frac{\ln 3}{\ln 2}}\right)
\end{equation}  
is of order $2\ln 3/\ln 2 \approx 3.17>3$.
After simple algebra one obtains the desired asymptotics for $\rho_x(k)$
and $\rho_y(k)$.
%\begin{equation} \label{rhoyk}
%\rho_y(k) = \frac{1}{\ln{2}} (2k+1)\pi + \frac{8}{3 \ln 2}
%\left( \frac{\pi^2}{8} \right)^{\frac{\ln 3}{\ln 2}}  
%\sin\left[ \frac{\ln 3}{\ln{2}} (2k+1) \pi \right]
%\varepsilon^{2\frac{\ln 3}{\ln 2}} + o\left(\varepsilon^{2\frac{\ln 3}{\ln 2}} \right)
%\end{equation}
%and
%\begin{eqnarray}
%\rho_x(k) & = & \frac{2}{\ln{2}}\ln\varepsilon +
%\left( -3 + \frac{2}{\ln 2} \ln \pi \right) + \frac{1}{\ln 2} \varepsilon
%\\ & & + \frac{1}{\ln 2} \left( \frac{1}{4} + \frac{7 \pi^2}{24} \right)
%\varepsilon^2 + \frac{1}{\ln 2} \left( \frac{1}{12}
%+ \frac{7 \pi^2}{24} \right) \varepsilon^3 \nonumber \\
%& & + \frac{8}{3 \ln 2} \left( \frac{\pi^2}{8} \right)^{\frac{\ln 3}{\ln 2}}  
%\cos\left[ \frac{\ln 3}{\ln{2}} (2k+1) \pi \right]
%\varepsilon^{2\frac{\ln 3}{\ln 2}} + o\left(\varepsilon^{2\frac{\ln 3}{\ln 2}}\right) . \label{rhoxk}
%\end{eqnarray}  
\end{proof}

\begin{figure}[t]
\centering
\includegraphics[clip,width=0.75\textwidth]{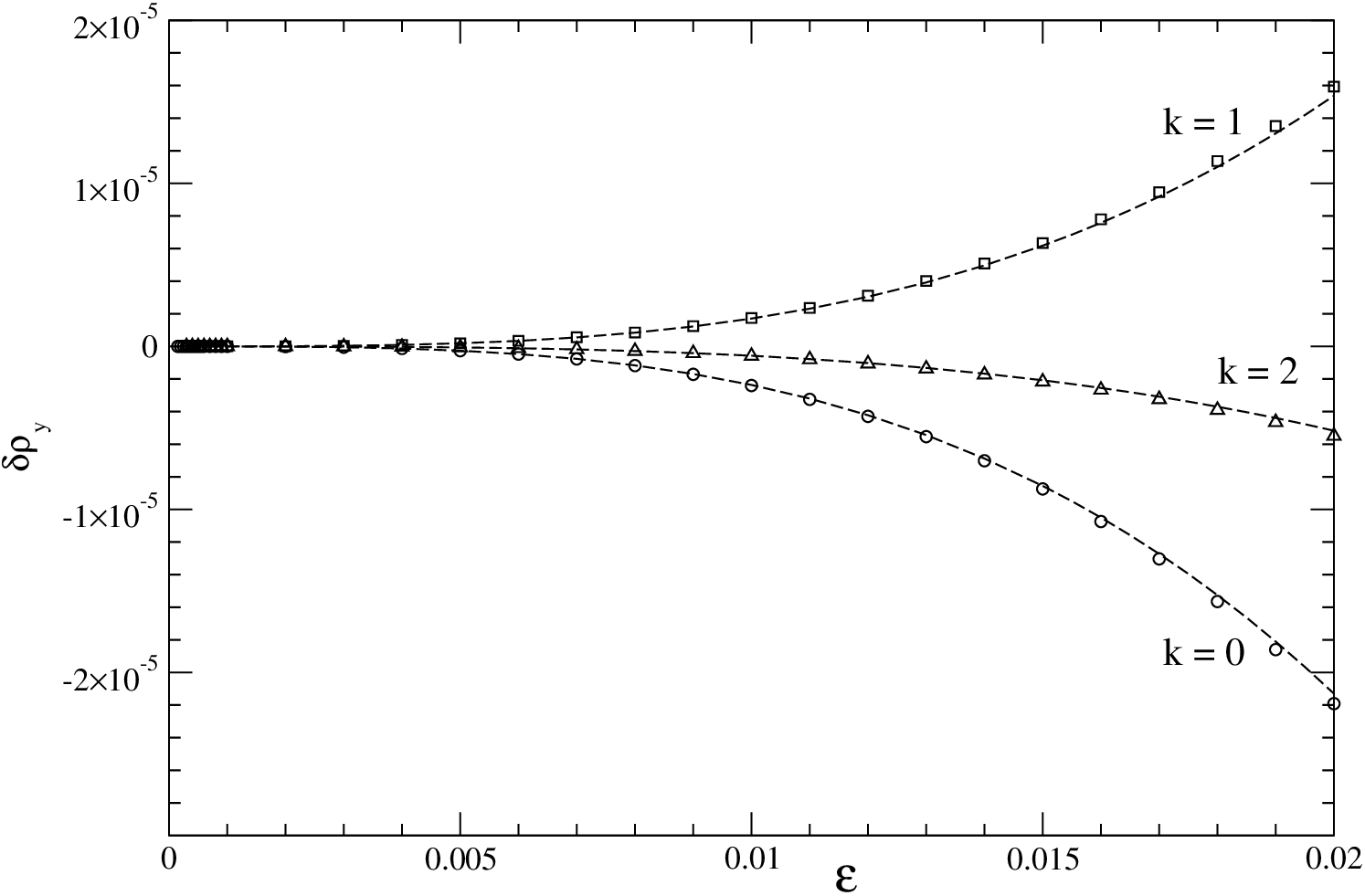}
\caption{Imaginary parts of the first three $(k=0,1,2)$ off-critical zeros
$E(\rho,\Delta)=0$.
The deviation $\delta\rho_y(k)$ is defined by (\ref{deltarhoy1}).
The numerical results are depicted by open circles $(k=0)$, squares $(k=1)$
and triangles $(k=2)$.
The plots of $\delta\rho_y(k)$ in the region of the small anisotropy parameter
$\varepsilon=1-\Delta$, anticipated to behave according to the asymptotic
formula (\ref{deltarhoy2}), are represented by the corresponding dashed curves.}
\label{roy-eps}
\end{figure}

\textbf{Comparison with our numerics.}
To check numerically our expansion in $\varepsilon$ for the imaginary parts of
the first three $(k=0,1,2)$ off-critical zeros, let us define the deviations
from their $\varepsilon=0$ values as follows
\begin{equation} \label{deltarhoy1}
\delta\rho_y(k) := \rho_y(k) - \frac{1}{\ln{2}} (2k+1)\pi .
\end{equation}
We know from \eqref{rhoyk} that the deviations are expected to behave in
the region of the small anisotropy parameter $\varepsilon\to 0^+$ as
\begin{equation} \label{deltarhoy2}
\delta\rho_y(k) = \frac{8}{3 \ln 2}
\left( \frac{\pi^2}{8} \right)^{\frac{\ln 3}{\ln 2}}  
\sin\left[ \frac{\ln 3}{\ln{2}} (2k+1) \pi \right]
\varepsilon^{2\frac{\ln 3}{\ln 2}} +o\left(\varepsilon^{2\frac{\ln 3}{\ln 2}} \right).
\end{equation}
The numerical results for $\delta\rho_y(k)$ are depicted by open circles
$(k=0)$, squares $(k=1)$ and triangles $(k=2)$ in Figure~\ref{roy-eps}.
It is seen that the numerical data fit perfectly the plots given by the
asymptotic formula (\ref{deltarhoy2}), represented by dashed curves,
for small values of $\varepsilon\leq 0.02$. 

\begin{figure}[t]
\centering
\includegraphics[clip,width=0.75\textwidth]{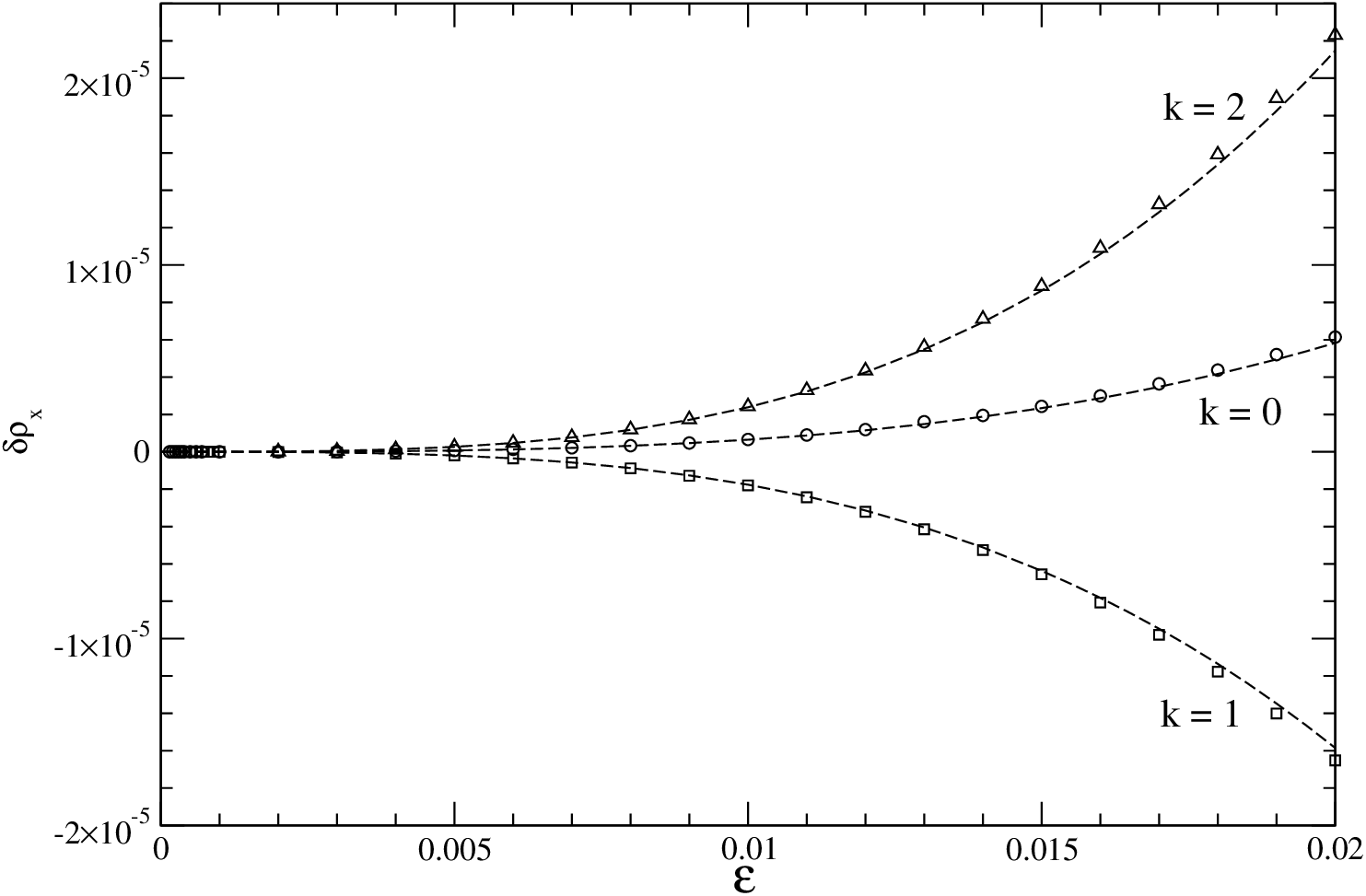}
\caption{Real parts of the first three $(k=0,1,2)$ off-critical zeros
$E(\rho,\Delta)=0$.
The deviation $\delta\rho_x(k)$ is defined by (\ref{deltarhox1}).
The numerical results are depicted by open circles $(k=0)$, squares $(k=1)$
and triangles $(k=2)$.
The plots of $\delta\rho_x(k)$ in the region of the small anisotropy parameter
$\varepsilon=1-\Delta$, anticipated to behave according to the asymptotic
formula (\ref{deltarhox2}), are represented by the corresponding dashed curves.}
\label{rox-eps}
\end{figure}

As concerns the real parts of the first three $(k=0,1,2)$ off-critical zeros,
we define the deviations as follows
\begin{eqnarray} 
\delta\rho_x(k) & := & \rho_x(k) -\frac{2}{\ln{2}}\ln\varepsilon -
\left( -3 + \frac{2}{\ln 2} \ln \pi \right) -\frac{1}{\ln 2} \varepsilon
\nonumber \\ & &
- \frac{1}{\ln 2} \left( \frac{1}{4} + \frac{7 \pi^2}{24} \right)
\varepsilon^2 - \frac{1}{\ln 2} \left( \frac{1}{12}
+ \frac{7 \pi^2}{24} \right) \varepsilon^3 . \label{deltarhox1}
\end{eqnarray}
It is obvious from (\ref{rhoyk}) that the deviations are anticipated
to behave for small values of anisotropy $\varepsilon\to 0^+$ as
\begin{equation} \label{deltarhox2}
\delta\rho_x(k) = \frac{8}{3 \ln 2}
\left( \frac{\pi^2}{8} \right)^{\frac{\ln 3}{\ln 2}}  
\cos\left[ \frac{\ln 3}{\ln{2}} (2k+1) \pi \right]
\varepsilon^{2\frac{\ln 3}{\ln 2}} +o\left( \varepsilon^{2\frac{\ln 3}{\ln 2}}\right).
\end{equation}  
The numerical results for $\delta\rho_x(k)$ are represented by open circles
$(k=0)$, squares $(k=1)$ and triangles $(k=2)$ in Figure~\ref{rox-eps}.
The numerical data fit very well the plots deduced from the
asymptotic formula (\ref{deltarhoy2}) (dashed curves).

\section*{Acknowledgements}
The support received from VEGA Grant No. 2/0092/21 
and Project EXSES APVV-20-0150 is acknowledged.

\end{document}